\documentclass[11pt,a4paper]{article}

\usepackage[T1]{fontenc}
\usepackage[utf8]{inputenc}
\usepackage[english]{babel}
\usepackage{color}
\usepackage{wrapfig}
\usepackage{placeins}
\usepackage{subfigure}
\usepackage{tabu}
\usepackage{fullpage}
\usepackage[squaren]{SIunits}
\usepackage{graphicx}
\usepackage{natbib}
\usepackage{tikz}
\usepackage{pgfplots}
\usepackage{epstopdf}
\usepackage{epsfig}
\usepackage{pifont}

\usepackage[colorlinks=true,bookmarks=false,linkcolor=blue,urlcolor=blue,citecolor=blue,breaklinks=true]{hyperref}
\usepackage{breakcites}
\usepackage{tikz}
\usepackage{tikz-qtree}
\usepackage{eurosym}
\usepackage{amsmath}
\numberwithin{equation}{section}
\usepackage{amssymb}
\usepackage{amsthm}
\usepackage{mathrsfs}
\usepackage{blkarray}
\usepackage{dsfont}
\usepackage{comment}
\usepackage{relsize}

\usepackage{rotating}
\newcommand{\minimize}{\textrm{minimize}}

\newcommand{\R}{\mathbb{R}}

\newcommand{\M}{\mathcal{M}}
\newcommand{\calM}{\mathcal{M}}

\newcommand{\calU}{\mathcal{U}}
\newcommand{\U}{\mathcal{U}}

\newcommand{\calO}{\mathcal{O}}
\newcommand{\red}{\textcolor{red}}
\definecolor{purple}{rgb}{0.74, 0.2, 0.64}

\newcommand{\calT}{\mathcal{T}}
\newcommand{\calN}{\mathcal{N}}
\newcommand{\D}{\mathrm{D}}

\newcommand{\p}{^{\perp}}
\renewcommand{\P}{\mathrm{P}}

\newcommand{\txt}{\text}
\newcommand{\transpose}{^\top}

\newcommand{\grad}{\mathrm{grad}}

\newcommand{\Hess}{\mathrm{Hess}}

\newcommand{\T}{\mathrm{T}}

\newcommand{\On}{\mathrm{O}(n)}
\newcommand{\SOn}{\mathrm{SO}(n)}
\newcommand{\SO}{\mathrm{SO}}

\newcommand{\Rn}{\mathbb{R}^n}
\newcommand{\RN}{\mathbb{R}^N}
\newcommand{\Rns}{\mathbb{R}^{n\times s}}

\newcommand{\inner}[2]{\left\langle{#1},{#2}\right\rangle}

\newcommand{\I}{\mathrm{I}}

\newcommand{\SURE}{\mathrm{SURE}}
\newcommand{\RMSE}{\mathrm{RMSE}}

\usepackage{bm}
\DeclareMathOperator*{\argmin}{arg\,min}

\newcommand{\range}{\operatorname{range}}

\newcommand{\Grass}{\mathrm{Grass}}
\usepackage{algorithm}
\usepackage{algpseudocode}
\usepackage{todonotes}

\newcommand{\aref}[1]{\hyperref[#1]{A\ref{#1}}}
\newcommand{\norm}[1]{\left\|#1\right\|}
\newcommand{\fronorm}[1]{\left\|#1\right\|_\mathrm{F}}

\usepackage{apxproof}
\newtheorem{theorem}			     {Theorem}	[section]

\newtheorem{definition}	         {Definition}[section]
\newtheorem{assumption} {A\ignorespaces}
\newtheorem{remark}			     {Remark}	[section]
\newtheoremrep{example}			     {Example}
\newtheoremrep{proposition} [theorem] {Proposition}

\usepackage{listings}
\usepackage{textcomp}
\definecolor{listinggray}{gray}{0.9}
\definecolor{lbcolor}{rgb}{0.9,0.9,0.9}
\usepackage{mathtools}
\mathtoolsset{showonlyrefs=false}

\newcommand{\noiseomega}{\mathlarger{\mathlarger{\omega}}}
\newcommand{\GrassNq}{\Grass(N,q)}
\newcommand{\RNq}{\R^{N\times q}}

\usepackage{authblk}

\title{Registration of algebraic varieties using Riemannian optimization}
\author[1,2]{Florentin Goyens}
\author[2]{Coralia Cartis}
\author[3]{St\'ephane Chr\'etien}
\date{\today}

\affil[1]{\textsc{LAMSADE} Universit\'e Paris Dauphine PSL, France}
\affil[2]{University of Oxford, UK, mathematical institute}
\affil[3]{National Physical Laboratory, UK\footnote{This work was generously supported by the  National Physical Laboratory (NPL), United Kingdom.}}
\affil[ ]{Corresponding author: \url{goyensflorentin@gmail.com} }

\begin{document}
\tabulinesep=1.2mm
\maketitle

\begin{abstract}
We consider the point cloud registration problem, the task of finding a transformation between two point clouds that represent the same object but are expressed in different coordinate systems. Our approach is not based on a point-to-point correspondence, matching every point in the source point cloud to a point in the target point cloud.
 Instead, we assume and leverage a low-dimensional nonlinear geometric structure of the data. Firstly, we approximate each point cloud by an algebraic variety (a set defined by finitely many polynomial equations). This is done by solving an optimization problem on the Grassmann manifold, using a connection between algebraic varieties and polynomial bases.
Secondly, we solve an optimization problem on the orthogonal group to find the transformation (rotation $+$ translation) which makes the two algebraic varieties overlap. We use second-order Riemannian optimization methods for the solution of both steps. 
Numerical experiments on real and synthetic data are provided, with encouraging results. Our approach is particularly useful when the two point clouds describe different parts of an objects (which may not even be overlapping), on the condition that the surface of the object may be well approximated by a set of polynomial equations. 
 The first procedure---the approximation---is of independent interest, as it can be used for denoising data that belongs to an algebraic variety. We provide statistical guarantees for the estimation error of the denoising using Stein's unbiased estimator. \newline
\textbf{keywords:} point cloud registration, Riemannian optimization, algebraic varieties.
\end{abstract}


\section{Introduction}

The problem of point cloud registration consists in finding a transformation between two point clouds that represent the same shape but are expressed in different coordinate systems. It appears for instance in computer vision~\citep{sharp2004multiview}, distributed approaches to molecular conformation~\citep{cucuringu2012eigenvector}, and sensor network localization~\citep{cucuringu2012sensor}. In this work, we consider instances where the shape of the point clouds can be represented by an algebraic variety (a set defined by finitely many polynomial equations). We develop a registration method based on manifold optimization which exploits the algebraic variety structure of the data.

Consider two algebraic varieties $V_1, V_2 \subset \Rn$ of the same degree $d$ (Definition~\ref{def:variety}). We assume the existence of a rigid transformation $\calT \colon\R^n \rightarrow \R^n$ that makes the varieties $V_1$ and $V_2$ overlap. 
\begin{assumption}\label{assu:variety_model}
There exists $Q\in \On$ and $a\in \R^n$ such that for all $x_1 \in V_1$, we have
\begin{equation}
\mathcal{T}(x_1) := Qx_1 + a \in V_2,
\label{eq:model}
\end{equation}
where $\On = \left\lbrace Q\in \R^{n\times n}: Q^\top Q  =\I_n\right\rbrace$ denotes the \emph{orthogonal group}.
\end{assumption}
Intuitively, $V_1$ and $V_2$ represent the same shape but in different coordinate systems, which are connected by a rotation and translation.

 A matrix $X\in \Rns$ represents a \emph{point cloud} in $\Rn$ and each column of $X$ is called a \emph{data point} or \emph{sample}. Let $M_1 \in \R^{n\times s_1}$ and $M_2 \in \R^{n\times s_2}$ be composed of respectively $s_1$ samples in $V_1$ and $s_2$ samples in $V_2$. Given $M_1$ and $M_2$, our task is to estimate $Q\in \mathrm{O}(n)$ and $a\in \R^n$ which define the transformation $\mathcal{T}$. The samples of $M_1\in \R^{n\times s_1}$ and $M_2\in \R^{n \times s_2}$ are independent and the two point clouds may have a different number of samples ($s_1 \neq s_2$). The goal is therefore not to establish a point-to-point matching, but to find a transformation such that the varieties $V_1$ and $V_2$ --- observed through the samples in $M_1$ and $M_2$ --- overlap as best as possible.

\begin{center}
\begin{tikzpicture}[scale =0.3]
\begin{scope}[shift = {(14,6)},scale=1,  rotate=-45,transform shape ]
\begin{axis}[hide axis,colormap/violet]
\addplot3[
	surf,
	domain=-2:2,
	domain y=-1.3:1.3,
] 
	{exp(-x^2-y^2)*x};
\end{axis}
\filldraw[fill=black] (1,2.5) circle (.1);
\filldraw[fill=black] (5,3.6) circle (.1);
\filldraw[fill=black] (5.4,2.9) circle (.1);
\filldraw[fill=black] (6,3.2) circle (.1);
\filldraw[fill=black] (6,3.2) circle (.1);
\filldraw[fill=black] (4.2,2) circle (.1);
\filldraw[fill=black] (4.5,2.5) circle (.1);
\filldraw[fill=black] (4,3) circle (.1);
\filldraw[fill=black] (3,3) circle (.1);
\filldraw[fill=black] (2,2.3) circle (.1);
\filldraw[fill=black] (1.7,2.8) circle (.1);
\filldraw[fill=black] (2.5,3.5) circle (.1);
\end{scope}
\node[above] at (3,4) {$M_1$};
\node[above] at (20,7) {$M_2$};
\draw[->,thick] (8,4) to [out = 45, in = 150] (15,7);
\node[above] at (11,7) {\large{$\mathcal{T} $ ?}};
\begin{axis}[hide axis,colormap/violet]
\addplot3[
	surf,
	domain=-2:2,
	domain y=-1.3:1.3,
] 
	{exp(-x^2-y^2)*x};
\end{axis}
\filldraw[fill=black] (1,2.5) circle (.1);
\filldraw[fill=black] (5,3.6) circle (.1);
\filldraw[fill=black] (5.4,2.9) circle (.1);
\filldraw[fill=black] (6,3.2) circle (.1);
\filldraw[fill=black] (6,3.2) circle (.1);
\filldraw[fill=black] (4.2,2) circle (.1);
\filldraw[fill=black] (4.5,2.5) circle (.1);
\filldraw[fill=black] (4,3) circle (.1);
\filldraw[fill=black] (3,3) circle (.1);
\filldraw[fill=black] (2,2.3) circle (.1);
\filldraw[fill=black] (1.7,2.8) circle (.1);
\filldraw[fill=black] (2.5,3.5) circle (.1);
\end{tikzpicture}
\end{center}

The situation described in~\aref{assu:variety_model} is a \emph{noiseless} setting, in which all the columns of $M_1$ belong exactly to the same algebraic variety. This is not likely to occur in practice, instead we expect that the data approximately belongs to an algebraic variety. It is this \emph{noisy} setting that we consider. Our strategy is to first tackle the problem of finding the algebraic variety that best approximates a point cloud. We do so in Section~\ref{sec:smoothing}. Each point cloud is approximated by an algebraic variety, and then we compute a rigid transformation between these two algebraic varieties (Section~\ref{sec:registration}). 

We model each of these two recovery problems (approximation by an algebraic variety and registration) as the minimization of a function defined on a combination of the following sets: the orthogonal group, the Grassmann manifold, and the Euclidean spaces $\Rns$ and $\Rn$. These optimization problems are instances of the problem class 
\begin{equation}\label{eq:manopt}
\min_{x\in \M} f(x)
\end{equation}
where $f\colon \M\to \R$ is a smooth nonconvex function defined on a Riemannian manifold $\M$. Hence, all the optimization problems that we formulate can be minimized using off-the-shelf Riemannian optimization methods~\citep{absil2008,boumal2023}, which are designed for problems of the form~\eqref{eq:manopt}. When the problems are nonconvex, there are no a priori guarantees to find the global minimizer of $f$ on $\calM$, and our numerical results show that local Riemannian optimization methods perform well on our test problems in Sections~\ref{sec:numerics_denoising} and~\ref{sec:numerics_registration}. We present a summary of the main existing approaches for registration in order to highlight the novelties in our contribution. 

\subsection{Related work}		
Several approaches have been devised for the various types of registration problems that appear in image and pattern analysis. Registration algorithms are classified into rigid and non-rigid approaches. Rigid registration methods find affine transformations---rotations and translations---that preserve distances. The non-rigid registration methods also consider nonlinear transformations. 
		
One of the early instances of registration in the literature is the (orthogonal) Procruste problem~\citep{schonemann1966generalized}. Given two matrices $A,B \in\Rns$, one tries to find the orthogonal matrix $Q^*\in \On$ that best matches $A$ to $B$ in Frobenius norm, that is,
\begin{align*}
Q^* = \argmin_{Q \in \On} \fronorm{QA -B}.
\end{align*}	
This problem is equivalent to finding the nearest orthogonal matrix to $BA\transpose$, 
\begin{align*}
Q^* = \argmin_{Q\in\On} \fronorm{Q - B A\transpose}, 
\end{align*}
which, given the singular value decomposition $B A \transpose = U\Sigma V\transpose$, is $Q^* = UV\transpose$. 

A natural extension is to consider a transformation which combines an orthogonal matrix with a translation. Consider two point clouds $M_1 = \{x_1, \dots, x_s\}$ and $M_2 = \{y_1, \dots, y_s\}$ in $\Rn$, respectively called the \emph{source} and the \emph{target} point clouds. One is obtained through a \emph{rigid transformation} of the other:
\begin{align}
y_k &= Q x_k + a && \text{for } k = 1,\dots,s,
\label{eq:rigid-transformation}
\end{align}
where $Q\in \mathrm{O}(n)$ is a rotation and $a\in \Rn$ is a translation. In this setting, \emph{exact point matching} is assumed: the two point clouds have the same number of samples and each sample matches with one and only one sample from the other point cloud. This leads to the following least squares problem
\begin{align}
(Q^*,a^*) = \argmin_{Q\in \On, a\in \Rn } \sum_{k =1}^s \norm{y_k - Q x_k - a}^2_2.
\label{eq:registartion-ls}
\end{align}
The optimization over $\On\times \Rn$ seems difficult since the set $\On$ is nonconvex. Remarkably, there exists a closed-form solution for the global minimizer of~\eqref{eq:registartion-ls}.~\citet{arun1987} show that the optimal orthogonal matrix is given by $Q^* = UV\transpose$ where $U\Sigma V\transpose$ is the SVD of 
\begin{align*}
\sum_{k=1}^s (x_k-x_c)(y_k - y_c)\transpose,
\end{align*}
with $x_c = (x_1 + \dots + x_s)/s$ and $y_c = (y_1 + \dots + y_s)/s$, the centroids of the two point clouds. The optimal translation is $a^* = y_c - Q^* x_c$. In the case of three or more point clouds to align, no closed-form solution is known. \citet{krishnan2005global} consider the case of multiple point clouds using a rigid transformation and exact point matching. They propose to solve a single optimization problem with one variable on $\SOn = \left\lbrace Q\in \R^{n\times n}: Q^\top Q  =\I_n, \mathrm{det}(Q)= +1\right\rbrace$ for each point cloud. \citet{chaudhury2015global} also consider an arbitrary number of point clouds and show how the least square formulation, which generalizes~\eqref{eq:registartion-ls} to several point clouds, can be relaxed into a convex program, with guarantees of tightness.

These approaches, as well as the Procruste problem, require knowledge of the \emph{point correspondence}: which point from the target corresponds to a given point in the source point cloud. This is unknown in a number of applications. For problems with exact point matching, the crux of the problem then resides in finding the correct point correspondence. The registration is then solved using a single SVD to find the global solution of~\eqref{eq:registartion-ls}. 

The iterative closest point (ICP) algorithm, introduced by~\citet{besl1992method}, addresses the issue of the unknown point correspondence. Assuming exact point matching, the ICP algorithm performs rigid registration in an iterative fashion. The method repeats the following two steps until the least squares error becomes smaller than some threshold: 
\begin{enumerate}
\item find a point correspondence using the nearest neighbour rule; \\
\item estimate a rotation and translation by solving~\eqref{eq:registartion-ls} using SVD.
\end{enumerate}
In the original version of ICP, the matching in step 1 is determined by the nearest neighbour criteria. For a point in $M_1$, its corresponding point is the closest point from $M_2$ in Euclidean norm. Note that this may not produce a one-to-one matching. Because the point correspondence changes throughout the iterations, it is difficult to establish theoretical guarantees of convergence for the classical ICP algorithm. The use of the least squares residual~\eqref{eq:registartion-ls} yields a method referred to as point-to-point ICP and is known to be sensitive to noise and outliers~\citep{bellekens2014survey}. 

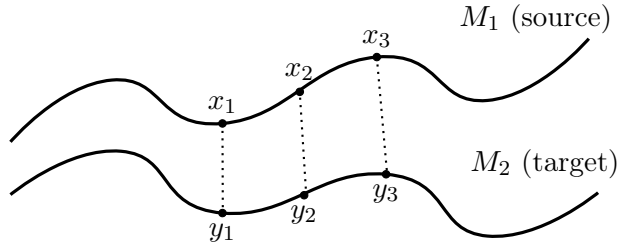
\begin{figure}[!htb]
\centering
\begin{tikzpicture}[scale=0.7]
\begin{scope}[rotate = 5]
\draw [very thick, black] plot [smooth, tension=0.9] coordinates {(0,0) (2,1)  (4,0) (7,1) (9,0) (11,1)};
\node[above] at (10,1) {$M_1$ (source)};
\filldraw [black] (4,0) circle (2pt);
\filldraw [black] (7,1) circle (2pt);
\filldraw [black] (5.5,0.47) circle (2pt);
\node[above] at (4,0) {$x_1$};
\node[above] at (5.5,0.5) {$x_2$};
\node[above] at (7,1) {$x_3$};
\draw[thick,dotted] (4,0) -- (3.85,-1.8);
\draw[thick,dotted] (5.5,0.5) -- (5.45,-1.5);
\draw[thick,dotted] (7,1) -- (7,-1.2);
\end{scope}
\begin{scope}[shift = {(0,-1)},scale=1, rotate = -5]
\draw [very thick,black] plot [smooth, tension=0.9] coordinates {(0,0) (2,1)  (4,0) (7,1) (9,0) (11,1)};
\node[above] at (10,1) {$M_2$ (target)};
\filldraw [black] (4,0) circle (2pt);
\filldraw [black] (7,1) circle (2pt);
\filldraw [black] (5.5,0.47) circle (2pt);
\node[below] at (4,0) {$y_1$};
\node[below] at (5.5,0.47) {$y_2$};
\node[below] at (7,1) {$y_3$};
\end{scope}
\end{tikzpicture}
\caption{Point-to-point ICP.}
\end{figure}

Numerous adaptations of the original ICP algorithm have been proposed, impacting all stages of the algorithm. 

The point-to-plane ICP (also called point-to-surface) is one such adaptation~\citep{chen1992object}. To reduce the algorithm's sensitivity to noise, step 2 minimizes the distance between a source point and the linear approximation of the point cloud at its target point (Figure~\ref{fig:p2p_icp}). 
\begin{figure}[!htb]
\centering
\includegraphics[scale=0.3]{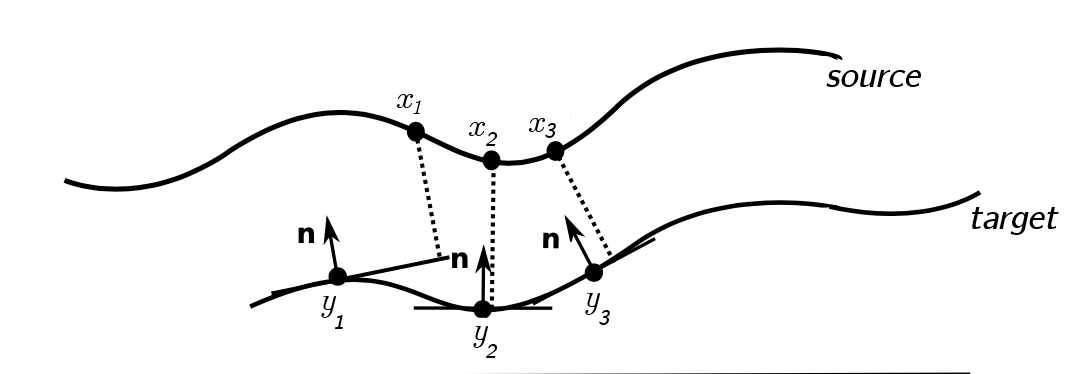}
\caption{Point-to-plane ICP~\citep{bellekens2014survey}.}
\label{fig:p2p_icp}
\end{figure}
This gives
\begin{align}
(Q^*,a^*) = \argmin_{Q\in \mathrm{O}(n), a\in \Rn } \sum_{k =1}^s | \inner{y_k - Q x_k - a}{ \mathbf{n}_k}|^2,
\label{eq:registartion-p2plane}
\end{align}
where $y_k$ is the corresponding point to $x_k$ and $\mathbf{n}_k$ is the normal vector to the linear approximation of the point cloud at $y_k$. This point-to-plane error metric does not have a closed-form solution and is minimized using nonlinear least squares methods. Although each iteration of the point-to-plane ICP is typically slower than an iteration of the point-to-point ICP, the former displays a faster convergence rate in practice~\citep{rusinkiewicz2001efficient}, as is predicted by theoretical analysis~\citep{pottmann2004registration}. 
 Plane-to-plane error metrics have also been proposed, which minimize the distance between the two tangent approximations around corresponding points. Generalized ICP allows to leverage the covariance matrix of the point clouds~\citep{segal2009generalized}. Variants of ICP are summarized in~\citep{rusinkiewicz2001efficient}. 
		 
In the non-rigid registration paradigm, the transformation is no longer restricted to the combination of a rotation and a translation. There exists a wide range of approaches for the non-rigid registration which we do not cover in depth. Some of the popular approaches for non-rigid registration include robust point matching~\citep{gold1998new}; thin plate spline robust point matching~\citep{chui2003new}; kernel correlation approach~\citep{tsin2004correlation} and coherent point drift~\citep{myronenko2010point}.
		 		 
The idea of learning the equations that define an algebraic variety also appears in~\citep{breiding2018learning}, where the authors investigate topological properties of the algebraic variety based on samples. \citet{fan2018nonlinear} consider the denoising of images using the polynomial kernel. 

\subsection{Contributions and outline of the paper}
We attempt to find a rigid transformation between two point clouds and assume that the target and source point clouds belong to the same algebraic variety (up to a change in coordinates). Our method does not assume that a point correspondence is given, nor does it attempt to compute one. This avoids the convergence issues that stem from the iterative  estimation of a point correspondence. The transformation that we compute aims that the image of the source by the rigid transformation belongs to the algebraic variety of the target point cloud. This conceptually aligns the algebraic varieties of the source and target point clouds. 


In order to reduce the sensitivity to noise and outliers, variants of the ICP algorithm minimize a point-to-plane or plane-to-plane metric using local linear approximations of the point clouds. The approach we consider can be labelled as a \emph{point-to-algebraic variety} matching, which does not minimize a distance between pairs of points or planes. 
Our residual is computed in a space of polynomial features, and measures the distance between a source point and the algebraic variety that defines the target point cloud. This also allows us to seamlessly consider two point clouds of different sizes and handle cases where the source point cloud only overlaps partially with the target point cloud, which is an issue for point-to-point methods.

 In Section~\ref{sec:polynomials}, we describe the connection between algebraic varieties and polynomial bases, which allow us to computationally represent algebraic varieties. For noisy data that does not belong to an algebraic variety, we have developed a technique to identify an approximation of each point cloud by an algebraic variety (Section~\ref{sec:smoothing}). This is then used to compute the rigid transformation (Section~\ref{sec:registration}). We show numerical results on synthetic examples and 3D medical scans that illustrate the efficiency and accuracy of our approach.  

\section{Preliminaries: algebraic varieties and polynomial features}
\label{sec:polynomials}
We describe how points that belong to an algebraic variety interact with a polynomial basis. An algebraic variety is a subset of $\Rn$ defined by finitely many polynomial equations. 
\begin{definition}[Algebraic variety model~\citep{cox1994ideals}]
\label{def:variety}
Let $\R_d[x]$ be the set of real-valued polynomials of degree at most $d$ over $\R^n$. A real (affine) \emph{algebraic variety} of degree $d$ is defined as the roots of a system of polynomials $P \subset \R_d[x]$:
\begin{equation*}
V(P) = \lbrace x\in \R^{n}: p(x) = 0 \txt{ for all } p \in P\rbrace.
\end{equation*}
We say that the matrix $X = \begin{bmatrix}
x_1 & x_2 & \cdots & x_s
\end{bmatrix} \in \Rns$ follows an algebraic variety model if every column of $X$ belongs to the same algebraic variety, i.e., $x_i \in V(P)$ for all $i = 1,\dots,s$.
\end{definition}
Let 
\begin{equation}\label{eq:N_features}
N = \begin{pmatrix}
n+d\\
d
\end{pmatrix} = \dfrac{(n+d)!}{d! \, n!}, 
\end{equation}
 be the dimension of $\mathbb{R}_d[x]$ and consider polynomials $\varphi_d^{(1)}, \dots,\varphi_d^{(N)}$ that form a basis of $\mathbb{R}_d[x]$. We define the \emph{polynomial features} as
\begin{equation}
\varphi_d \colon \R^n \to \R^{N} \colon  \varphi_d(x) = \begin{pmatrix}
\varphi_d^{(1)}(x)\\
\varphi_d^{(2)}(x)\\
\vdots \\
\varphi_d^{(N)}(x)\\
\end{pmatrix}.
\label{eq:polynomial_features}
\end{equation}
For $X\in \Rns$, we obtain the matrix of polynomial features $\Phi_d(X)$ by applying $\varphi_d$ to each column of $X$,
\begin{align*} 
 \Phi_d(X) = \begin{bmatrix}
\varphi_d(x_1) & \cdots & \varphi_d(x_s)
\end{bmatrix}\in \R^{N\times s}.
\end{align*}
Let the algebraic variety $V(P)\subset \R^n$ be defined by a set of $q$ linearly independent polynomials $P=\lbrace p_1,\dots,p_q\rbrace$, where each polynomial $p_i$ is at most of degree $d$. Consider the features $\Phi_d(X)$ of a matrix $X\in \R^{n \times s}$ whose columns $x_1, \dots, x_s$ belong to $V(P)$ and assume the following:
 \begin{assumption}\label{assu:samples}
The dimensions of the matrix $X\in\Rns$ satisfy $s \geq N-q$.
\end{assumption}
Under~\aref{assu:samples}, we have
\begin{align}\label{eq:variety_null_phi}
x_i &\in  V(P) \textrm{ for all } i = 1,\dots,s && \textrm{ if and only if } &\Phi_d(X)^\top U&=0,
\end{align}
where the $q$ columns of $U \in \R^{N\times q}$ contain the coefficients of the polynomials $p_1,\dots,p_q$ in the basis $\varphi_d$. In this way, we view any matrix $U\in \R^{N\times q}$ as representing $q$ polynomials in the basis $\varphi_d$. The zero set of these polynomials is an algebraic variety, and hence  we associate the matrix $U\in \RNq$ with that algebraic variety. 

This algebraic variety is only a function of $U$ through $\range(U)$. Take $ U_2 \in \R^{N\times q}$ such that $\range( U_2) = \range(U)$ (there exists $W\in \R^{q\times q}$ such that $U_2= U W$). From~\eqref{eq:variety_null_phi}, the algebraic variety defined by $U_2$ is also $V(P)$. Thus, the algebraic variety $V(P)$ depends on the $q$-dimensional subspace $\calU := \range(U) \subset \RN$ but does not depend on the specific matrix $U$ chosen as a basis for that subspace. Hence, it makes sense to use the notation $V_\calU$ for the algebraic variety $V(P)$.

The space of all $q$-dimensional subspaces in $\RN$ is called the Grassmann manifold---written $\Grass(N,q)$--- and is a smooth Riemannian manifold. Any $\U\in \GrassNq$ therefore defines a unique algebraic variety. A point $\U\in \Grass(N,q)$ is represented on a computer by a matrix $U\in \R^{N \times q}$ such that $\mathrm{range}(U) = \U$. 
 
\section{Approximation by an algebraic variety (denoising)}
\label{sec:smoothing}
Consider a point cloud $M\in \Rns$ and its corrupted version
\begin{align}
\hat{M} = M + \noiseomega,
\label{eq:Mhat}
\end{align}
where the columns of $M$ belong to an algebraic variety $V(P)$ (Definition~\ref{def:variety}) and the noise $\noiseomega\in \R^{n\times s}$ has small Frobenius norm compared to $M$: $\fronorm{\noiseomega} \ll\fronorm{M}$.  The goal of denoising is to recover $M$ from $\hat{M}$ as accurately as possible, without knowing the polynomials that define $V(P)$ (Figure~\ref{fig:Mhat}). The distribution of the noisy perturbation $\noiseomega$ may be unknown. In computing $M$, we want to find the algebraic variety that best approximates the point cloud $\hat M$. The noise $\noiseomega$ may be interpreted as the model mismatch between the raw data (e.g medical scans) and an algebraic variety.

\begin{figure}[!h]
\begin{center}
\begin{tikzpicture}[scale=0.6]
\begin{axis}[hide axis,colormap/violet]
\addplot3[
	surf,
	domain=-2:2,
	domain y=-1.3:1.3,
] 
	{exp(-x^2-y^2)*x};
\end{axis}
\filldraw[fill=black] (1,2.5) circle (.1);
\filldraw[fill=red] (1,2.6) circle (.1);
\filldraw[fill=black] (5,3.6) circle (.1);
\filldraw[fill=red] (5,3.8) circle (.1);
\filldraw[fill=black] (5.4,2.9) circle (.1);
\filldraw[fill=red] (5.4,3.1) circle (.1);
\filldraw[fill=black] (6,3.2) circle (.1);
\filldraw[fill=black] (6,3.2) circle (.1);
\filldraw[fill=red] (6,3.6) circle (.1);
\filldraw[fill=black] (4.2,2) circle (.1);
\filldraw[fill=red] (4.2,2.2) circle (.1);
\filldraw[fill=black] (4.5,2.5) circle (.1);
\filldraw[fill=red] (4.5,2.7) circle (.1);
\filldraw[fill=black] (4,3) circle (.1);
\filldraw[fill=black] (3,3.2) circle (.1);
\filldraw[fill=black] (2,2.3) circle (.1);
\filldraw[fill=black] (1.7,2.8) circle (.1);
\filldraw[fill=black] (2.5,3.5) circle (.1);
\filldraw[fill=red] (4,3.5) circle (.1);
\filldraw[fill=red] (3,3.4) circle (.1);
\filldraw[fill=red] (2,2.5) circle (.1);
\filldraw[fill=red] (1.7,3) circle (.1);
\filldraw[fill=red] (2.5,3.7) circle (.1);
\node at (0,3) {$M$};
\node at (0.5,4) {\red{$\hat{M}$}};
\end{tikzpicture}
\end{center}
\caption{Denoising of an algebraic variety}
\label{fig:Mhat}
\end{figure}
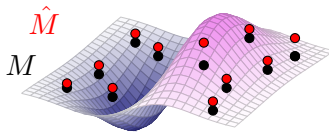

\subsection{Formulation as an optimization problem}

We start with an estimate of $q$, the number of linearly independent polynomials that define the algebraic variety $V(P)$. In order to estimate $V(P)$, we use that every subspace $\calU \in \GrassNq$ defines an algebraic variety $V_\calU$. From~\eqref{eq:variety_null_phi}, the subspace which defines the algebraic variety $V(P)$ belongs to the null space of $\Phi(M)\transpose$. It follows from simple linear algebra that $\mathrm{null}\left(\Phi(M)\transpose \right) = \left( \range \Phi(M)\right)\p$. So we aim to find a subspace orthogonal to $\range \Phi(M)$, that is, a subspace $\calU$ such that $\P_\calU\Phi_d(M)=0$, where $\mathrm{P}_{\U} = UU\transpose \in \R^{N\times N}$ denotes the orthogonal projection on the subspace $\U$ with $\mathrm{range}(U) = \U$ and $U\transpose U = \I_N$. Let $\{p_1, \dots, p_q\}$ be the polynomials that define the algebraic variety $V_{\calU}$, whose coefficients are given by $U = \begin{bmatrix}
u_1 & u_2 & \cdots & u_q
\end{bmatrix} \in \R^{N\times q}$ in the basis $\varphi_d$. For $X\in \Rns$, it follows 
\begin{align}\label{eq:residual_denoising}
\fronorm{\P_{\U} \Phi_d(X)}^2 = \fronorm{U\transpose \Phi_d(X)}^2 = \sum_{i=1}^s \sum_{j=1}^q  ((u_j)\transpose \varphi_d(x_i) )^2 = \sum_{i=1}^s\sum_{j=1}^q  p_j(x_i)^2.
\end{align}
 We minimize~\eqref{eq:residual_denoising} under the constraint that the variable $X$ is not ``far away'' from the noisy input $\hat M$. Given some $\eta>0$, an estimate of the noise level, we
\begin{align}\label{eq:p2_smoothing}
\underset{\substack{\U\in \GrassNq \\
X\in \Rns}}{\minimize}
 \fronorm{\P_\U \Phi_d(X)}^2  
 &&  \text{subject to} &&
\fronorm{X-\hat{M}}^2 \leq \eta.
\end{align}
In~\eqref{eq:p2_smoothing}, the cost function is nonconvex but smooth. It expresses the least square error of the polynomial system that defines the algebraic variety $V_{\calU}$, calculated for every column of $X$.
The above formulation can be made unconstrained in the variable $X$; for some $\lambda \geq 0$,
\begin{align}\label{eq:lambda_prob}
\underset{\substack{\U\in \GrassNq \\
X\in \Rns}}{\minimize}
 \fronorm{ \P_\U \Phi_d(X)}^2  +\lambda \fronorm{X-\hat{M}}^2.
\end{align}

It can be shown that there exists a $\lambda \geq 0$ such that~\eqref{eq:p2_smoothing} and~\eqref{eq:lambda_prob} are equivalent. Through the lens of Riemannian optimization, Problem~\eqref{eq:lambda_prob} can be viewed as  the unconstrained minimization of a function defined on the manifold \mbox{$\M = \R^{N\times s} \times \Grass(N,q)$}. This general framework allows to use off-the-shelf Riemannian optimization methods~\citep{boumal2023}.

Problem~\eqref{eq:lambda_prob} is nonconvex but there are good initial guesses for the variables X and $\U$, which are respectively $\hat M$ and the last $q$ (left) singular vectors of $\Phi_d(\hat M)$. In Section~\ref{sec:numerics_denoising}, we observe numerically that solving~\eqref{eq:lambda_prob} with second-order Riemannian optimization methods is robust and appears to find the global minimum despite the nonconvexity. In order to solve~\eqref{eq:lambda_prob}, we do not assume knowledge of the noise level $\eta$. To find a suitable value for the regularization parameter $\lambda$, the problem is first solved with $\lambda = 10^{-6}$ (a small value) and then solved again several times with increasing values of $\lambda$. Each problem is initialized with the solution of the previous solve (warm starting). This process is stopped when a good trade-off between the values of $\fronorm{\P_{\calU}\Phi_d (X)}$ and $\fronorm{X-\hat M}$ is reached.

When the residual vanishes, the points belong exactly to an algebraic variety. Equation~\eqref{eq:variety_null_phi} implies that
\begin{align}
 \fronorm{\P_\U \Phi_d(X)}&=0 & \textrm{ if and only if } &&  x_i \in V_{\calU} \textrm{ for all } &i = 1,\dots, s
 \label{eq:variety_identification_smoothing}
\end{align}
where $x_1, \dots, x_s$ denote the columns of $X\in \Rns$. In that case, the last $q$ singular values of $\Phi_d(X)$ are zero and $V_\calU$ is defined by the subspace corresponding to the last $q$ singular vectors $\Phi_d(X)$, which may be computed by a singular value decomposition. Recovering the algebraic variety of the data---despite the presence of noise---is the first part of the registration process in Section~\ref{sec:registration}, where we estimate a transformation such that two algebraic varieties overlap. 

An important practical concern is to estimate $q$, the number of linearly independent polynomials which define the algebraic variety. For hypersurfaces of dimension $n-1$, one polynomial equation defines the surface and $q=1$. In the section on numerical results~\ref{sec:numerics_denoising}, we discuss the choice of a suitable degree $d$ for the model. The next section provides an estimate of the accuracy of the recovery in the case of Gaussian noise, which gives an estimate of the error at a solution which is obtained for a given degree. 

\subsection{Statistical error estimation}
\label{sec:stats}
In this section, we use Stein's unbiased risk estimate~\citep{stein1981estimation} to give guarantees on the error of our denoising procedure. Assume that the noise perturbation $\noiseomega$ in~\eqref{eq:Mhat} follows a normal distribution. That is, for some $\sigma >0$, we have $\noiseomega_{ij} \sim \mathcal{N}(0, \sigma^2)$ for every entry $(i,j)$ of $\noiseomega$. For a given $\hat M$, let $\left(X^*(\hat M), \U^*(\hat M)\right)$ be an isolated local minimizer of~\eqref{eq:lambda_prob}. We view $X^*(\hat M)$ as an estimator of $M$ which depends on the random variable $\noiseomega$. The expected error is defined as $R= \mathbb{E}_\noiseomega \! \left[\fronorm{M - X^*(\hat M)}^2\right]$. Based on~\citep{stein1981estimation}, we have 
\begin{equation}
\hat R =  \fronorm{\hat M - X^*(\hat M)}^2 -ns \sigma^2  + 2 \sigma^2 \sum_{i=1}^n \sum_{j=1}^s \dfrac{\partial X^*_{ij}(\hat M)}{\partial \hat M_{ij}},
\label{eq:stein}
\end{equation}
as an unbiased estimate for $R$, i.e. $\mathbb{E}(\hat R) = R$. The estimate $\hat R$ is called \emph{Stein's unbiased risk estimate}  (SURE). The quantity $ \sum_{i=1}^n \sum_{j=1}^s \partial X^*_{ij}(\hat M)/\partial \hat M_{ij}$ is known as the \emph{divergence} of the estimator $X^*(\hat M)$ and describes the sensitivity of the solution to the input data. It measures the complexity of the model, that is, its tendency to overfit the data~\citep[Eq. 33]{ghojogh2019theory}. 

We propose to use SURE as an alternative to cross validation to evaluate the estimation error and detect the possibility of overfitting. For polynomials of degree $2$, which we use in the numerical examples, the model is simple and the likelihood of overfitting is small. Estimating the generalization error with formula~\eqref{eq:stein} is most useful for several values of model parameters like the degree $d$. 

To compute the divergence in the context of Problem~\eqref{eq:lambda_prob}, we use a version of the implicit function theorem for functions defined on manifolds. For a smooth map $F\colon \M_1\times \M_2 \to \M_3$, we write $\D_1 F$ and $\D_2 F $ for the differential of $F$~\citep[Definition 3.34]{boumal2023}, with respect to its first and second argument, respectively.

 \begin{theorem}[\citep{abraham2012manifolds}, Prop. 3.3.13]\label{thm:implicit_function}
Let $\M_1, \M_2, \M_3$ be manifolds. Let $F\colon \M_1 \times \M_2 \to \M_3$ be  smooth and let $(x_0, y_0) \in \M_1 \times \M_2$ with $F(x_0, y_0) =0$. If $\D_2 F(x_0, y_0)$ is invertible, there exists open neighbourhoods $V_1$ of $x_0$ in $\M_1$ and $V_2 $ of $y_0$ in $\M_2$, and a smooth function $g:V_1\to V_2$ such that for all $x\in V_1$, $F(x,g(x)) = 0$. In addition,
\begin{equation*}
\D g(x_0) = - \left( \D_2 F(x_0, y_0)\right)^{-1} \D_1 F(x_0, y_0).
\end{equation*}
\end{theorem}
Recall that the root mean square error (RMSE) is defined by 
\begin{equation}\label{eq:rmse}
\RMSE = \fronorm{M - X^*}/\sqrt{ns}.
\end{equation}
\begin{proposition}
Let the noise $\noiseomega \sim \calN(0, \sigma^2)$ for some given $\sigma\geq 0$ and let   $f(X,\U; \hat M)$ denote the cost function of~\eqref{eq:lambda_prob}. Consider $\left(X^*(\hat M), \U^*(\hat M)\right)$, an isolated local minimizer of $f$ for a given $\hat M$. Stein's Unbiased risk estimate of the RMSE is given by
\begin{align}\label{eq:sure}
\SURE := \sqrt{\hat R/ns}
\end{align}
where $\hat R$ is computed as in~\eqref{eq:stein} with 
\begin{align}
\dfrac{\partial X^*(\hat M)}{\partial \hat M} = - \left[ \nabla^2_{XX} f(X,\U; \hat  M)\right]^{-1}  \dfrac{\partial}{\partial \hat M}  \! \left(\nabla_X f(\hat M, g(\hat M);\hat M) \right).
\label{eq:stein_result}
\end{align}
where
 $$\dfrac{\partial}{\partial \hat M_{ij}}  \! \left(\nabla_X f(\hat M, g(\hat M);\hat M) \right)\!= -2\lambda \Delta_{ij} $$
and $\Delta_{ij}\in \R^{n\times s}$ is zero except for entry $(i,j)$ equal to $1$. 
\end{proposition}
\begin{proof}
For clarity, we write $(X^*,\U^*)$ for $\left(X^*(\hat M), \U^*(\hat M)\right)$. Recalling that $\M = \Rns \times \GrassNq$, we define $F\colon \mathbb{R}^{n\times s} \times \M \to \T \M$ by
\begin{equation}
F(\hat M, X, \U) =\grad_\M f(X, \calU;\hat M) =  \begin{bmatrix}
\nabla_X f(X, \U; \hat  M) \\
\grad_\U f( X, \U; \hat  M)
\end{bmatrix}.
\label{eq:F}
\end{equation}
where $ \nabla_X f(X, \U;\hat M) \in \R^{n \times s}$ is the Euclidean gradient of $f$ with respect to $X \in \R^{n \times s}$ and $\grad_\calU f(X, \U;\hat M) \in \T_\calU \GrassNq$ is the Riemannian gradient which belongs to the tangent space of $\GrassNq $ at $\U$. 
First-order necessary optimality conditions for Problem~\eqref{eq:lambda_prob} can be stated as 
\begin{align*}
F(\hat M, X, \U) = 0.
\end{align*}
We apply Theorem~\ref{thm:implicit_function} to the function $F$, with the identification of the variables $\hat M\in \M_1 = \Rns$ and $z=(X,\U) \in \M_2 = \M = \Rns \times \GrassNq$. The theorem applies to derivatives as abstract objects defined by tangent vectors. In particular, affine connections~\citep[Definition 5.1]{boumal2023}---including the Riemannian connection---satisfy the axioms of derivatives. When applying this theorem to the function $F$, we view the derivative of a vector field as the covariant derivative associated to the Riemannian connection. Therefore, the derivative of the map $F$ with respect to its second variable $z\in \M$, is given by the Riemannian Hessian of $f$, that is, $\D_2 F(\hat M, X, \U)  = \Hess_\M f(X,\U; \hat  M) $. Therefore, if $\Hess_\M f(X^*, \U^*; \hat  M)$ is invertible, there exists an open set $V_1\in \Rns$ and a map $ g\colon  V_1 \to \M$  such that $g(\hat M) = (X^*, \U^*)$ and $F(\hat M, g(\hat M)) =0$. Also, 
\begin{equation}
\dfrac{\partial g(\hat M)}{\partial \hat M} = \dfrac{\partial}{\partial \hat M} \left((X^*(\hat M), \U^*(\hat M)\right) = 
- \left[ \Hess_\M f(X^*, \U^*; \hat  M)\right]^{-1} \left[    \dfrac{\partial F(\hat M, g(\hat M))}{\partial \hat M}\right].
\label{eq:dg}
\end{equation}
If $(X^*,\U^*)$ is an isolated minimizer of~\eqref{eq:lambda_prob}, the matrix $\Hess f(X^*,\U^*)$ is positive definite and therefore invertible. In order to compute $\partial F(\hat M, g(\hat M))/\partial \hat M$, we see from the definition of $f(X,U;\hat M)$ in~\eqref{eq:lambda_prob} that $\grad_\U f(X,U;\hat M)$ does not depend on $\hat M$ and that $\nabla_X f(\hat M, g(\hat M);\hat M) $ depends on $\hat M$ through the term $2\lambda (X - \hat M)$. Thus, we have
\begin{align}
\dfrac{\partial}{\partial \hat M}  \!\left(\grad_\U f(\hat M, g(\hat M);\hat M) \right)\!= 0
\label{eq:g_zero}
\end{align}
and
 $$\dfrac{\partial}{\partial \hat M_{ij}}  \! \left(\nabla_X f(\hat M, g(\hat M);\hat M) \right)\!= -2\lambda \Delta_{ij} $$
 with 
 where $\Delta_{ij}\in \R^{n\times s}$ has entry $ij$ equal to one and the other entries are zero. Using~\eqref{eq:g_zero}, the upper block of~\eqref{eq:dg} reads 
 \begin{align*}
\dfrac{\partial X^*(\hat M)}{\partial \hat M} = - \left[ \nabla^2_{XX} f(X,\U; \hat  M)\right]^{-1}  \dfrac{\partial}{\partial \hat M}  \! \left(\nabla_X f(\hat M, g(\hat M);\hat M) \right).
 \end{align*}
This gives~\eqref{eq:stein_result}. Equation~\eqref{eq:sure} follows from~\eqref{eq:stein} and the definition of the RMSE.
\end{proof}

\subsection{Numerical results for denoising of algebraic varieties}
\label{sec:numerics_denoising}
In this section, we numerically evaluate the performance of our denoising method. Our code for the denoising of algebraic varieties can be found online at \url{https://github.com/flgoyens/variety-denoising}, currently with implementations in Matlab and Python. We use the Manopt~\citep{manopt} and Pymanopt~\citep{townsend2016pymanopt} toolboxes for the Matlab and Python versions of the code, respectively.

The Riemannian trust-region solver~\citep[Algorithm 6.3]{boumal2023} in the toolbox is our algorithm of choice; it is applied to Problem~\eqref{eq:lambda_prob} with default parameters. The trust-region algorithm underlying this solver has been shown to have fast local rate of convergence (at least superlinear, near a non-degenerate minimizer~\citep{absil2007trust}). Furthermore, under suitable global smoothness assumptions, this algorithm is guaranteed to converge to a first- or second-order critical point of the problem from an arbitrary initial guess~\citep{boumal2019global}, with a rate that can be quantified as follows: for a given accuracy parameter $\varepsilon$, it takes at most $\calO(\varepsilon^{-2})$ to generate an $\varepsilon$-approximate first-order critical point, while an $\varepsilon$-approximate second-order critical point is found within most $\calO(\varepsilon^{-3})$ iterations. 

The numerical examples in this section were generated using the Matlab version of our code, with the termination condition that the norm of the gradient in~\eqref{eq:lambda_prob} becomes smaller than $10^{-6}$. For an output of the algorithm $(X^*,\U^*)$, we define
\begin{align*}
\textrm{RESIDUAL} =  \fronorm{\P_{\U^*} \Phi_d(X^*)}^2,
\end{align*}
which is interpreted using~\eqref{eq:residual_denoising} as the least squares error of the system of polynomial equations that define the algebraic variety $V_{\U^*}$. The RESIDUAL tells us how close the columns of $X^*$ are to some algebraic variety. The RMSE tells us how close $X^*$ is to the original matrix $M$~\eqref{eq:rmse}. We begin with randomly generated examples. We use a monomial basis of degree $d=2$ for the feature map $\varphi_d$. 
\subsubsection*{Synthetic examples}

\paragraph*{Example 1: Denoising a circle}
 We generate a point cloud $\hat M \in \R^{2\times 150}$ as a noisy corruption of a set of points on the unit circle using $\sigma = \{10^{-3},10^{-2}, 10^{-1}, 2\cdot 10^{-1}\}$ for the standard deviation of the Gaussian noise ($\hat{M}$ in red in Figure~\ref{fig:circle}). We notice that the solution $X^*$ (blue) reached by the optimization algorithm is close to the original circle even for a visually large noise. For $\sigma = 10^{-2}$ and $\sigma = 2\cdot 10^{-1}$, the RESIDUAL reaches values of $10^{-9}$ and $10^{-5}$, respectively. The solver's output $X^*$ is therefore very well approximated by an algebraic variety. This example also makes it clear that the approach is more general than a classical polynomial approximation, as the approximation does not need to be the image of a polynomial function. Table~\ref{tab:circle} reports the RMSE~\eqref{eq:rmse} and SURE estimate of the RMSE~\eqref{eq:sure}. SURE accurately predicts the RMSE and that the RMSE increases proportionally with the standard deviation of the noise. 
 \begin{table}[!h]
\centering
\begin{tabu}{|c|c|c|c|c|}
\hline 
Noise $\sigma$ & $10^{-3}$ & $10^{-2}$ & $10^{-1}$ & $2\cdot 10^{-1}$\\ 
\hline 
SURE & $5.49 \cdot 10^{-4}$ & $6.76 \cdot 10^{-3}$  & $5.69 \cdot 10^{-2}$  & $1.22 \cdot 10^{-1}$ \\ 
\hline 
RMSE & $6.77 \cdot 10^{-4}$ &  $8.53 \cdot 10^{-3}$ & $7.52 \cdot 10^{-2}$ & $1.44 \cdot 10^{-1}$\\ 
\hline 
\end{tabu} 
\caption{SURE~\eqref{eq:sure} and RMSE~\eqref{eq:rmse} for denoising of a circle}
\label{tab:circle}
\end{table}

\begin{figure}
\centering
\subfigure[$\sigma = 10^{-2}$]{
\includegraphics[width = 0.45\textwidth]{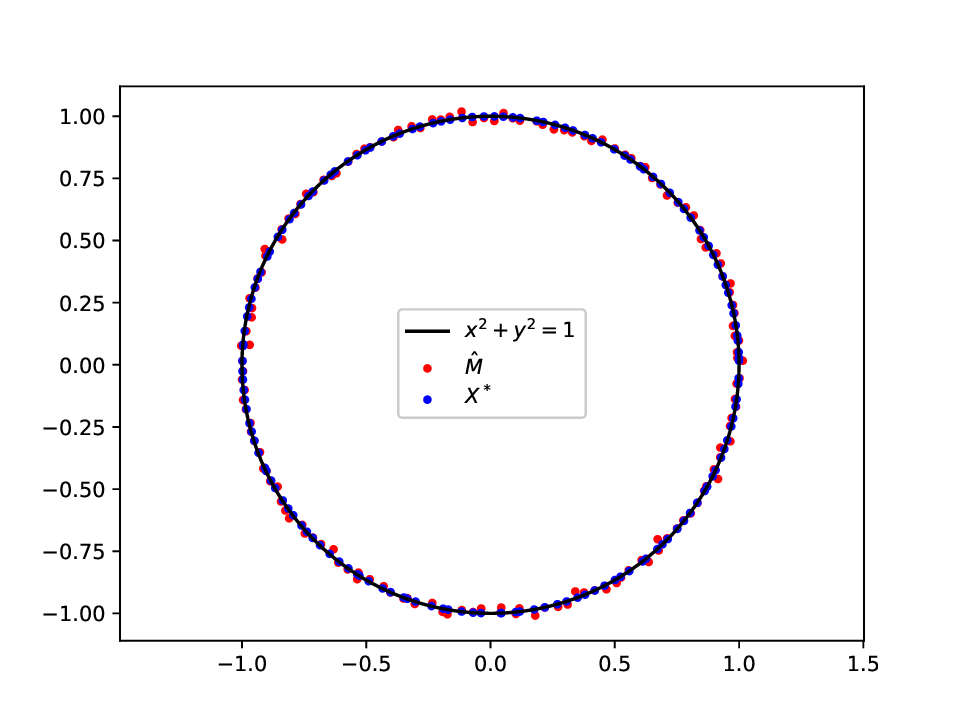}
}
\subfigure[$\sigma = 10^{-1}$]{
\includegraphics[width = 0.45\textwidth]{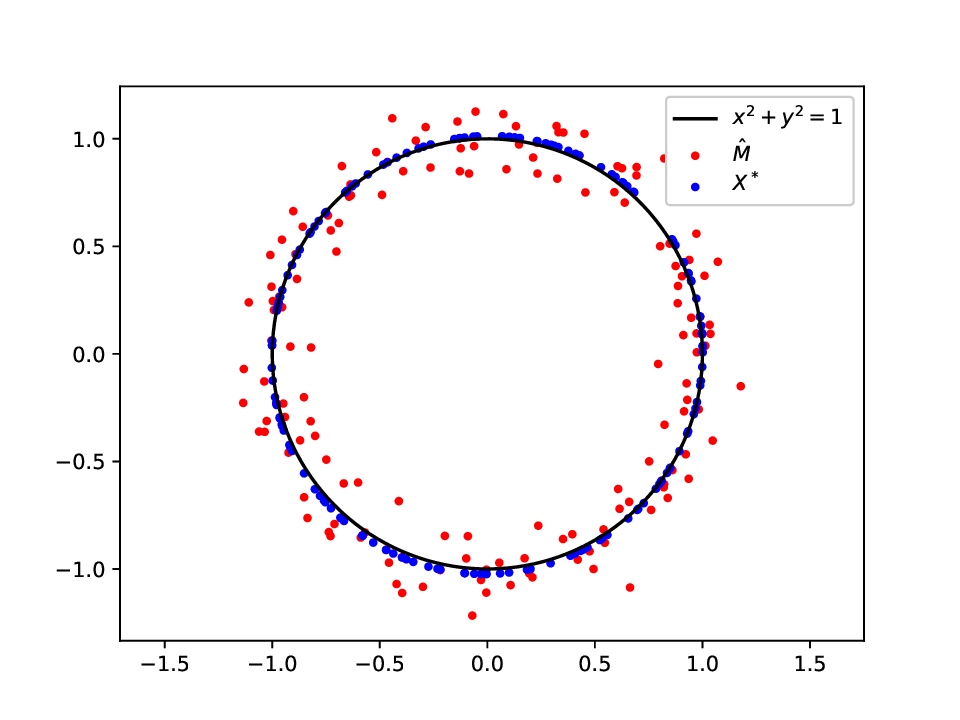}
}
\subfigure[$\sigma = 2\cdot 10^{-1}$]{
\includegraphics[width = 0.45\textwidth]{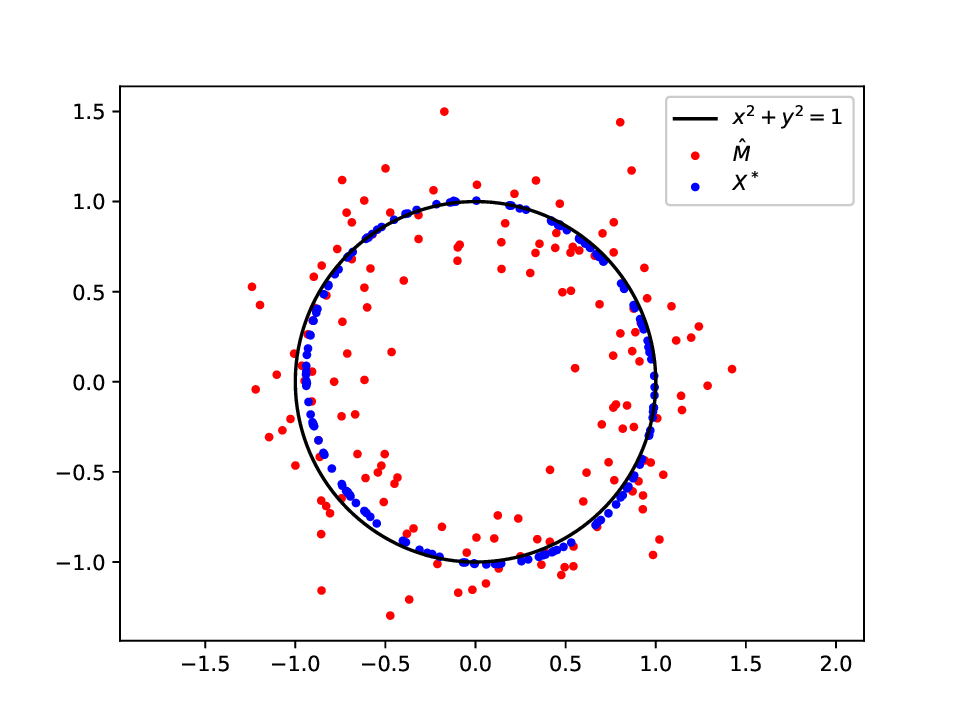}
}
\caption{Denoising a circle.}
\label{fig:circle}
\end{figure}


\paragraph*{Example 2: Denoising a union of two subspaces}
In Figure~\ref{fig:cross}, we denoise a point cloud which is near the union of two subspaces with $\sigma = \{10^{-3},10^{-2}, 10^{-1}\}$ for the standard deviation of the noise. This shows that the algebraic variety doesn't need to be a smooth set. It is an algebraic variety described by polynomial equations of degree $2$, hence using polynomials of degree $2$ is sufficient. (The union of $k$ subspaces is an algebraic variety described by polynomial equations of degree $k$.) For $\sigma = 10^{-2}$, the output satisfies $ \textrm{RESIDUAL} = 3\cdot10^{-10}$. Table~\ref{tab:cross} reports the RMSE and the SURE estimate of the RMSE.
\begin{table}[!h]
\centering
\begin{tabu}{|c|c|c|c|}
\hline 
Noise $\sigma$ & $10^{-3}$ & $10^{-2}$ & $10^{-1}$  \\ 
\hline 
SURE & $7.05 \cdot 10^{-4}$ & $ 6.32 \cdot 10^{-3} $  & $7.04  \cdot 10^{-2}$ \\ 
\hline 
RMSE & $7.10 \cdot 10^{-4}$ &  $6.98 \cdot 10^{-3} $ & $ 7.37 \cdot 10^{-2}$ \\ 
\hline 
\end{tabu} 
\caption{SURE for denoising of a union of subspaces}
\label{tab:cross}
\end{table}


\begin{figure}
\centering
\subfigure[$\sigma = 10^{-2}$]{
\includegraphics[width = 0.45\textwidth]{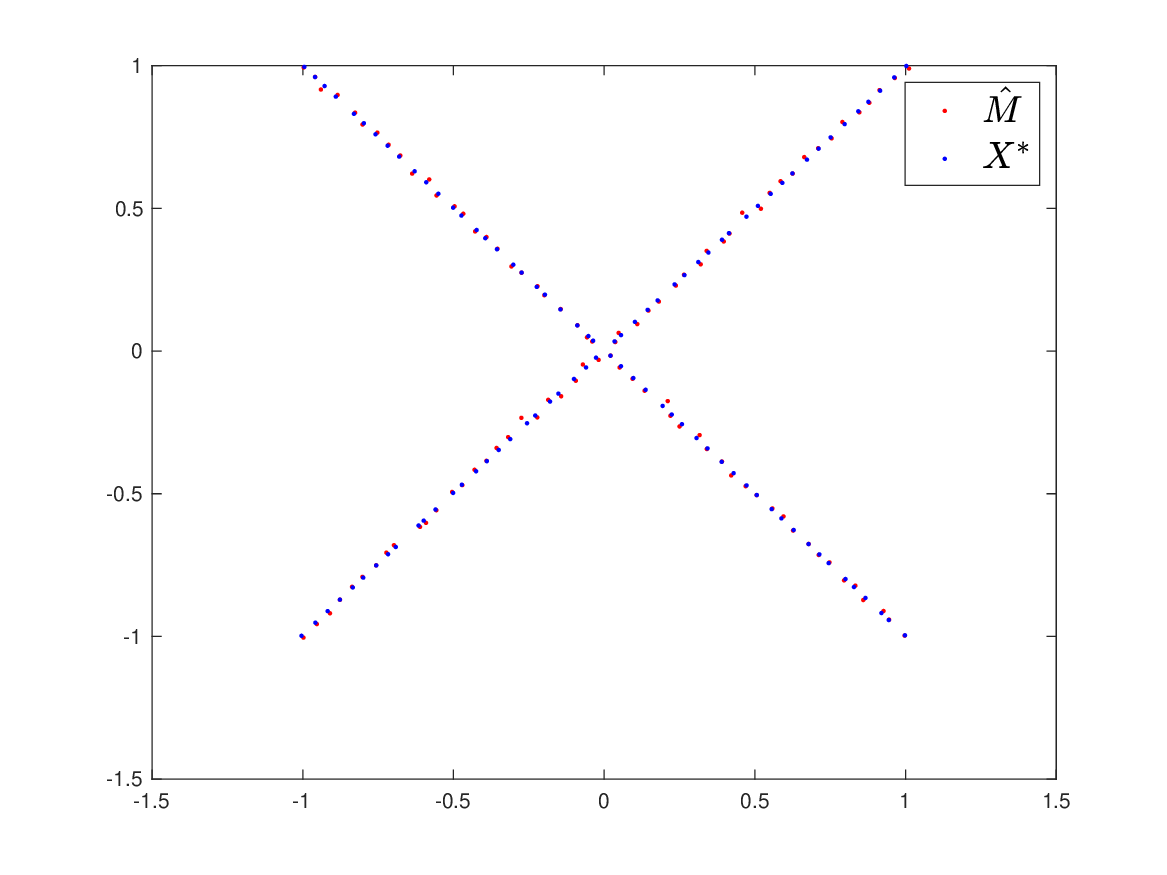}
}
\subfigure[$\sigma = 10^{-1}$]{
\includegraphics[width = 0.45\textwidth]{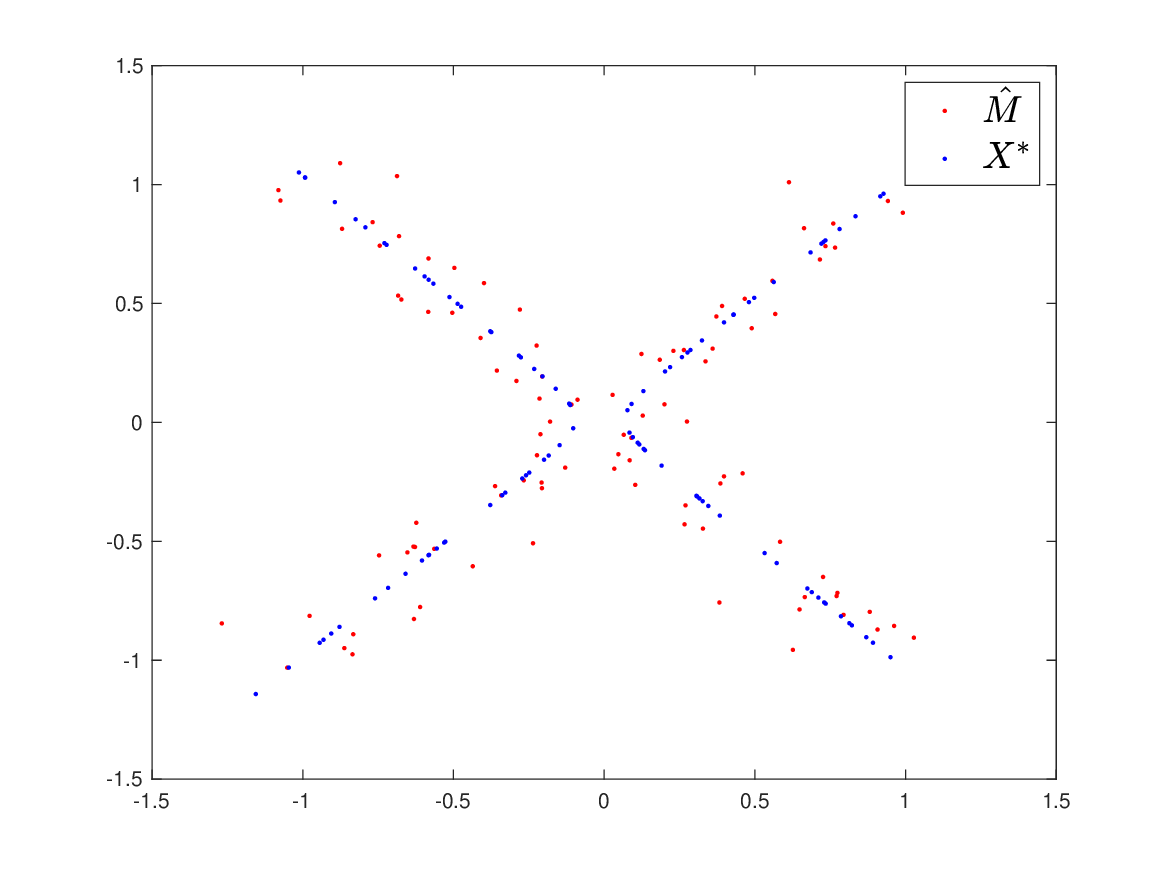}
}
\caption{Denoising a nonsmooth algebraic variety.}
\label{fig:cross}
\end{figure}

\subsubsection*{Dental Tomography Scans}  
Our research was originally motivated by an industrial partnership with the National Physical Laboratory (NPL), which concerned 3D dental scans (X-ray computed tomography scans). Each scan contains several millions of data points. It is desirable to remove the noise and outliers before they can be processed. Figure~\ref{fig:scans} shows an example of the dental scans available. After random subsampling---which helps visualization---it is composed of $2048$ data points in $\R^3$. 

\begin{figure}[!htb]
\centering
\includegraphics[width = 0.3\textwidth]{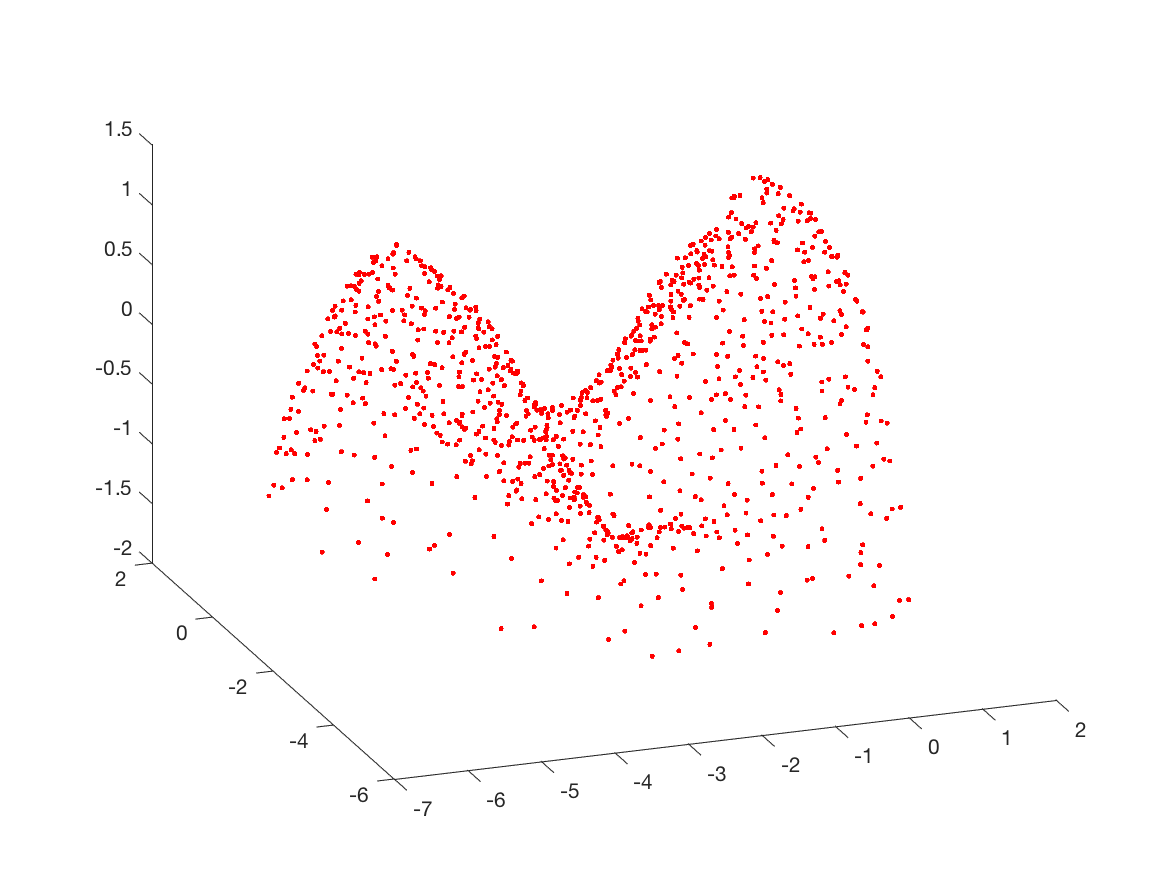}
\caption{Tomographic dental scans (subsampled)}
\label{fig:scans}
\end{figure}

For data sets such as dental scans, it may not be obvious which degree of the feature map will yield the best result. In theory, increasing the degree should give a model with more degrees of freedom and a better fit.
Unfortunately, the dimension of the optimization problem $N \approx n^d/d!$ blows up as $d$ increases, even for moderate values of $n$. The number of data points necessary also increases with the degree, since $s > N-q$ is required~\aref{assu:samples}. Additionally, ill-conditioning may appear when the degree $d$ becomes large, this is especially true for monomial basis because $\Phi_d(X)$ is a multivariate Vandermonde matrix. 
Previous works using the monomial kernel for high-rank matrix completion problems restricted themselves to degrees two or three~\citep{ongie2017algebraic,fan2018nonlinear,goyens2022nonlinear}. We try values of $d$ going from $1$ to $5$ and choose the one that yields the best result, which is usually $d=2$ in practice. 

\paragraph*{Example 3: approximation of a dental scan by an algebraic variety} Figure~\ref{fig:denoise-tooth} shows an XCT scan in red and the output of our denoising algorithm in blue. The value RESIDUAL has been reduced from the order of 1 to the order of $10^{-3}$ by the algorithm, and $X^*$ is therefore closer to an algebraic variety than $\hat M$. Since no noise was artificially added, the value of $\sigma$ for the noise that represents the mismatch between the real dental scan and the algebraic variety model is unknown. For $\sigma = 10^{-2}$, Stein's estimation of the RMSE gives $\SURE = 0.21$. 
\begin{figure}[!htb]
\centering
\includegraphics[scale=0.3]{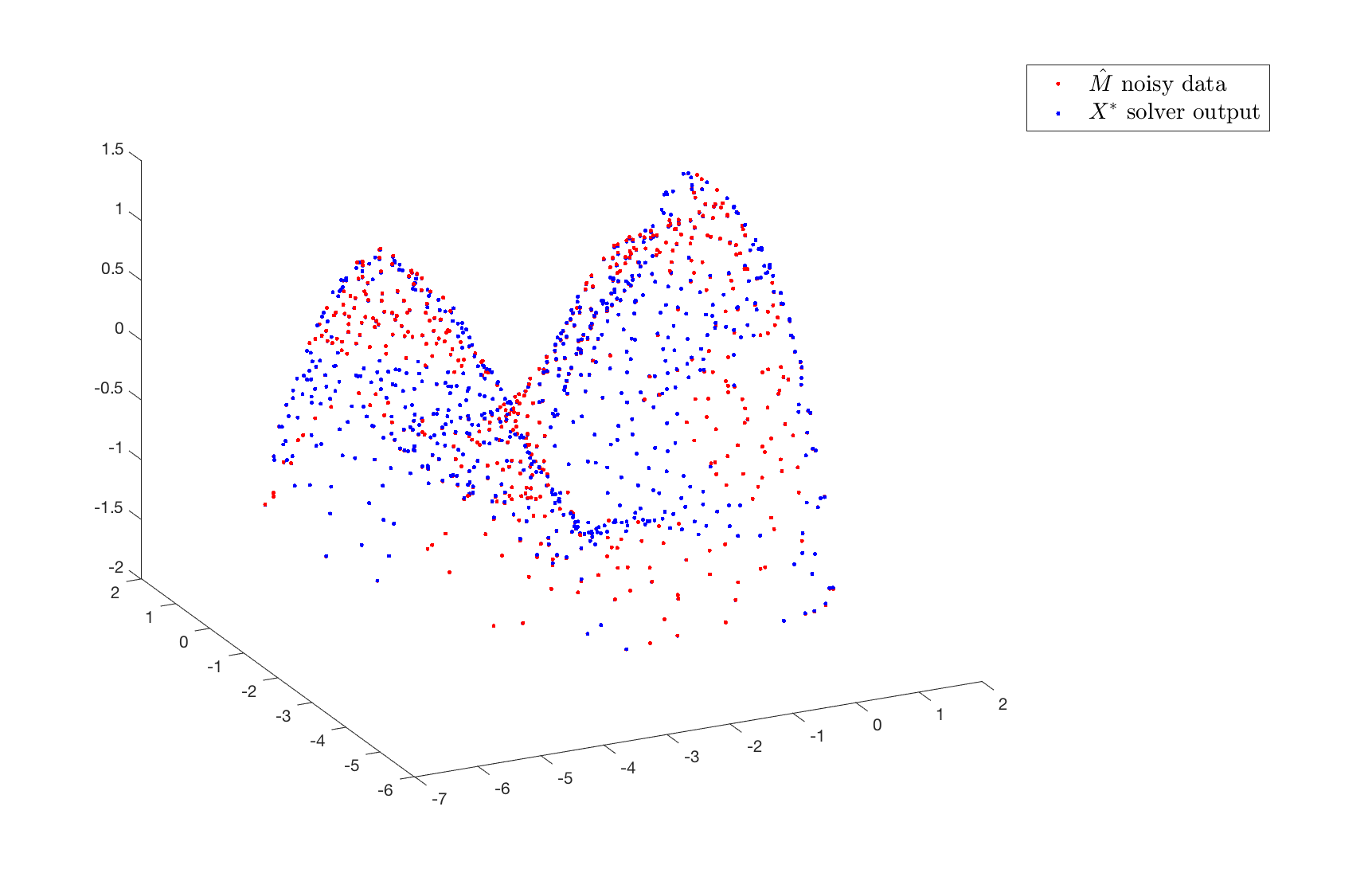}
\caption{Original point cloud and denoised version}
\label{fig:denoise-tooth}
\end{figure}

\FloatBarrier
\section{Registration of algebraic varieties}
\label{sec:registration}
We return to the registration problem presented in the introduction. Consider two matrices $M_1\in \R^{n\times s_1}$ and $M_2\in \R^{n\times s_2}$ that contain data points belonging to two algebraic varieties $V_1$ and $V_2$ that overlap through a rigid transformation (\aref{assu:variety_model}). Assume further that the algebraic varieties $V_1, V_2$ have degree at most $d$ and are described by $q$ linearly independent polynomials, for some $0<q\leq N$; and that $M_1$ and $M_2$ both satisfy~\aref{assu:samples}, so that $\min(s_1, s_2) \geq N-q$. We are given two point clouds $\hat{M}_1$ and $\hat{M}_2$, which are approximate versions of $M_1$ and $M_2$:
\begin{align*}
\hat M_1 &= M_1 + \noiseomega_1\\
\hat M_2 &= M_2 +\noiseomega_2,
\end{align*}
where $\noiseomega_1, \noiseomega_2$ are noisy perturbations, with $\fronorm{\noiseomega_1}\ll \fronorm{M_1}$ and $\fronorm{\noiseomega_2}\ll \fronorm{M_2}$. Our goal in this section is to estimate the rigid transformation $\mathcal{T}\colon \R^n \to \R^n $ which makes $V_1$ and $V_2$ overlap, or, more precisely, such that $\mathcal{T}(M_1)$ belongs to $V_2$.

\subsection{Formulation as an optimization problem}
 To achieve this, we first attempt to recover $M_1$ and $M_2$ from $\hat M_1$ and $\hat M_2$, by solving twice the denoising problem~\eqref{eq:lambda_prob}, as described in Section~\ref{sec:smoothing}. This yields estimates $X_1\in \Rns$ and $X_2\in \Rns$, along with subspaces $\U_1 \in \GrassNq$ and $\U_2\in \GrassNq$ which define the algebraic varieties that $X_1$ and $X_2$ (approximately) belong to. 
 Then, we estimate a transformation $\calT$ such that $\calT(X_1)$ belongs to the algebraic variety defined by $\calU_2$. 
 
Let $x_1^{(1)}, x_2^{(1)} ,  \cdots  x_{s_1}^{(1)}$ denote the columns of $X_1$ and let $p_1, \dots, p_q$ denote the polynomials that define $V_{\calU_2}$. To find a rotation $Q\in \SO(n)$ and a vector $a\in \Rn$ such that the points $\mathcal{T}\left(x_i^{(1)}\right) = Q x_i^{(1)} + a $ belong to $V_{\calU_2}$ for all $1\leq i\leq s_1$, we  
\begin{align}\label{eq:cost_registration}
\underset{Q,a}{\minimize}  \sum_{i=1}^{s_1}\sum_{j=1}^q  p_j(Q x^{(1)}_i + a)^2.
\end{align}
The distinction with the least squares minimization~\eqref{eq:registartion-ls} appears clearly. We do not minimize a point-to-point distance; instead, we are fitting points to an algebraic variety and want each column of $X_1$---after application of the rigid transformation---to satisfy the equations that define $V_{\calU_2}$. Using~\eqref{eq:residual_denoising}, we write~\eqref{eq:cost_registration} using the variable $\calU_2\in \GrassNq$ and the matrix of polynomial features:
\begin{align*}
 \sum_{i=1}^{s_1}\sum_{j=1}^q  p_j(Q x^{(1)}_i + a)^2 =\sum_{i=1}^{s_1} \norm{ \P_{\U_2}  \varphi_d(Q x^{(1)}_i + a)}_2^2 = \fronorm{\P_{\U_2}  \Phi_d(Q X_1 + a\mathbf{1}_{1\times s_1})}^2,
\end{align*}
where $\mathbf{1}_{1\times s_1}$ is a row vector of size $s_1$ that is full of ones. This leads to the following two-step strategy to identify the transformation between $V_1$ and $V_2$.
\paragraph*{Step 1: Algebraic variety identification and denoising}
\begin{equation}
(X_1, \U_1) := \left\lbrace
\begin{aligned}
& \underset{X,\U}{\argmin}
& & \fronorm{ \P_{\U}\Phi_d\left(X\right) }^2 + \lambda \fronorm{ X - \hat M_1}^2   \\
& & &  \U \in \Grass(N,q).
\end{aligned}
\right.
\label{eq:smoothing_noisy_regi1}
\end{equation}
\begin{equation}
(X_2, \U_2) := \left\lbrace
\begin{aligned}
& \underset{X,\U}{\argmin}
& & \fronorm{ \P_{\U}\Phi_d\left(X\right) }^2 + \lambda \fronorm{ X - \hat M_2}^2   \\
& & &  \U \in \Grass(N,q).
\end{aligned}
\right.
\label{eq:smoothing_noisy_regi2}
\end{equation}
\paragraph*{Step 2: Registration}
\begin{equation}
(Q^*, a^*):= 
 \left\lbrace
\begin{aligned}
& \underset{Q,a}{\min}
& & \fronorm{\P_{\U_2} \Phi_d\left(Q X_1 + a \mathbf{1}_{1\times s_1}\right) }^2   \\
& \text{subject to} & & Q \in \SO(n) \\
& & &  a \in \R^n.
\end{aligned}
\right.
\label{eq:regi_noisy}
\end{equation}
For details on solving Step 1, see Section~\ref{sec:smoothing}. The registration step~\eqref{eq:regi_noisy} is nonconvex, and it is also a smooth optimization problem on a manifold that we minimize using Riemannian optimization~\citep[Algorithm 6.3]{boumal2023}. We initialize the transformation $\calT$ with a random rotation and random vector $a$ of unit norm. If the transformation $\calT$ includes a reflection as well as a rotation, the optimization should not be done over $\SOn$ but the other connected component of $\On$, with matrices of determinant $-1$.
\begin{remark}
Another possibility would be to compute the registration in a single step, an optimization problem over the variables $(\calU, Q, a)$ that concatenates both point clouds. We found that this is harder to solve in practice. Performing steps 1 and 2 separately is suitable because we are able to estimate the algebraic variety from the samples of only one of the point clouds.
\end{remark}
\begin{remark}
 Interestingly, this two-step approach is asymmetric. If one point cloud is included in the other, the larger point cloud must be used as the target ($M_2$) and the smaller one as the source ($M_1$).
 \end{remark}

\subsection{Numerical results for registration}
\label{sec:numerics_registration}
Our code for the registration problem is available in Python at \url{https://github.com/flgoyens/variety-registration}. In the numerical examples that follow, the default implementation of the Riemannian trust-region algorithm from PyManopt was used where the full Riemannian Hessian is computed through automatic differentiation~\citep{townsend2016pymanopt}. This section illustrates the efficiency of our approach to identify a rigid transformation between two quadratic curves or surfaces (synthetic test problems), or two  dental scans. We use the monomial basis of degree $d=2$ for the feature map $\varphi_d$. We set the dimension $q=1$, since the points clouds are hyper surfaces of dimension $n-1$. 

For the synthetic examples, we take the samples in $M_2$ as randomly generated points that belong to the quadratic $x \mapsto  x^2$. The samples in $M_1$ come from a rotated and translated version of that curve, for some random rotation and translation that we attempt to recover. We then add random matrices of Gaussian noise $\noiseomega_1, \noiseomega_2 \sim \mathcal{N}(0, \sigma^2)$, which yields $\hat M_1$ and $\hat M_2$. The quality of the solution returned by the solver is assessed by the residual 
\begin{align*}
\textrm{RESIDUAL} = \fronorm{\P_{\U_2} \Phi_d\left(Q^* X_1 + a^*\mathbf{1}_{1\times s_1}\right) }^2,
\end{align*}
which is interpreted using~\eqref{eq:cost_registration} as the residual of the polynomial equations that define an algebraic variety. We emphasize that the measure $\fronorm{Q^* X_1 + a^*\mathbf{1}_{1\times s_1} - X_2}^2$ is meaningless. We are not trying to perform a point-to-point matching, which is even undefined for point clouds of different sizes. In the figures below, the upper image shows the two point clouds of input, and the lower image shows the output of the algorithm after registration. 

These numerical results illustrate that our approach successfully finds an accurate rigid transformation in these simple cases. Due to the nonconvexity of $\SOn$ and the random initialisation of the algorithm, the solver sometimes fails to find the correct transformation and appears to find a local minimum. We then restart Problem~\eqref{eq:regi_noisy} with a new initial rotation and translation, and report that this usually finds the global minimum under five attempts. This is likely helped by the fact that registration problems are most often considered in $\R^2$ or $\R^3$, and the dimension of the search space is small. A better initialization of the transformation would lower the chances to find a suboptimal local minimizer.  

\paragraph*{Example 5}
In Figure~\ref{fig:regi_noiseless_3d}, the data is on a quadratic surface in $\R^3$ (zero noise added). The input data has dimensions $M_1 \in \R^{3\times 200}$, $M_2 \in \R^{3\times 200}$. The result satisfies $\textrm{RESIDUAL}\approx 10^{-7}$. 
\begin{figure}[!htb]
\includegraphics[width = 1\textwidth]{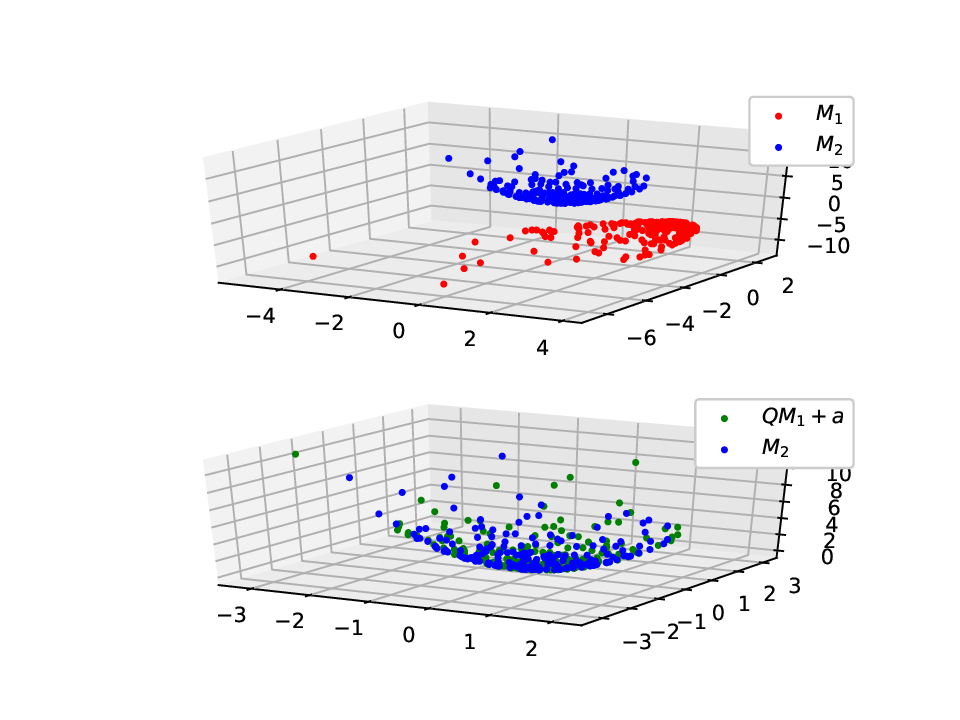}
\caption{Registration for a quadratic surface in $\R^3$.}
\label{fig:regi_noiseless_3d}
\end{figure}

\paragraph*{Example 6}
In Figure~\ref{fig:registration_noise_sigma2}, the standard deviation of the noise is $\sigma = 5\cdot 10^{-2}$ and the algorithm's output satisfies $\textrm{RESIDUAL} \approx 10^{-4}$. We see that despite the presence of noise in the original data, the transformation is well estimated to create an overlap.
\begin{figure}[!htb]
\centering
\includegraphics[width = 0.7\textwidth]{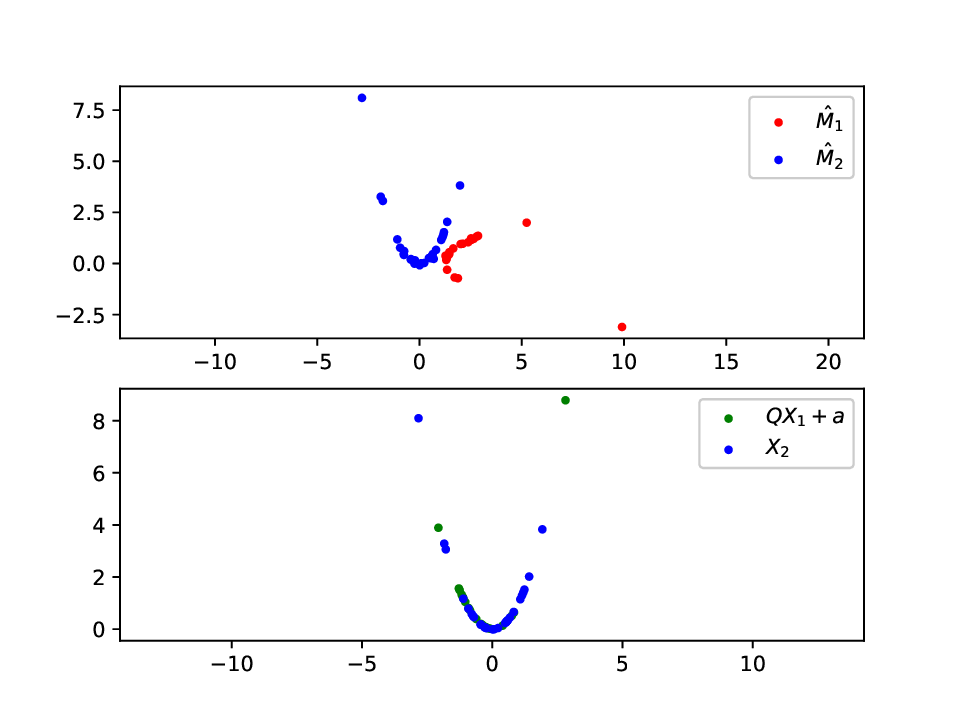}
\caption{Noisy registration with $\sigma = 5\cdot 10^{-2}$}
\label{fig:registration_noise_sigma2}
\end{figure}

\paragraph*{Example 7}
In Figure~\ref{fig:registration_noise_sigma1}, the magnitude of the noise is increased to $\sigma = 10^{-1}$ and the output satisfies $\textrm{RESIDUAL}  \approx 10^{-2}$. The transformation is again estimated correctly, with the residual increasing proportionally to the magnitude of the noise. 
\begin{figure}[!htb]
\centering
\includegraphics[width = 0.7\textwidth]{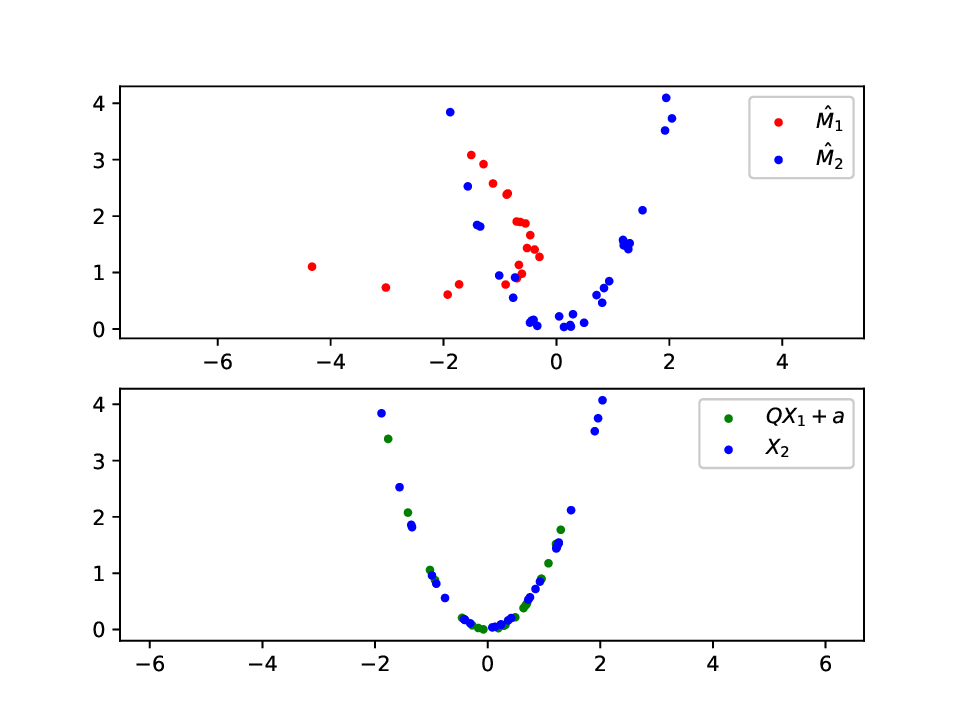}
\caption{Noisy registration with $\sigma = 10^{-1}$}
\label{fig:registration_noise_sigma1}
\end{figure}

Next we show examples where there is a partial overlap between the source and target point clouds. 

\paragraph*{Example 8: partial overlap}
In Figures~\ref{fig:registration_noise2_overlap} and~\ref{fig:registration_noise1_overlap}, the point cloud $\hat M_1$ only overlaps with a small part of $\hat M_2$.  The noise levels are of $\sigma = 10^{-2}$ and $\sigma= 10^{-1}$ and give RESIDUAL values of $10^{-4}$ and $10^{-2}$, respectively. 

\begin{figure}[!htb]
\centering
\includegraphics[width = 0.7\textwidth]{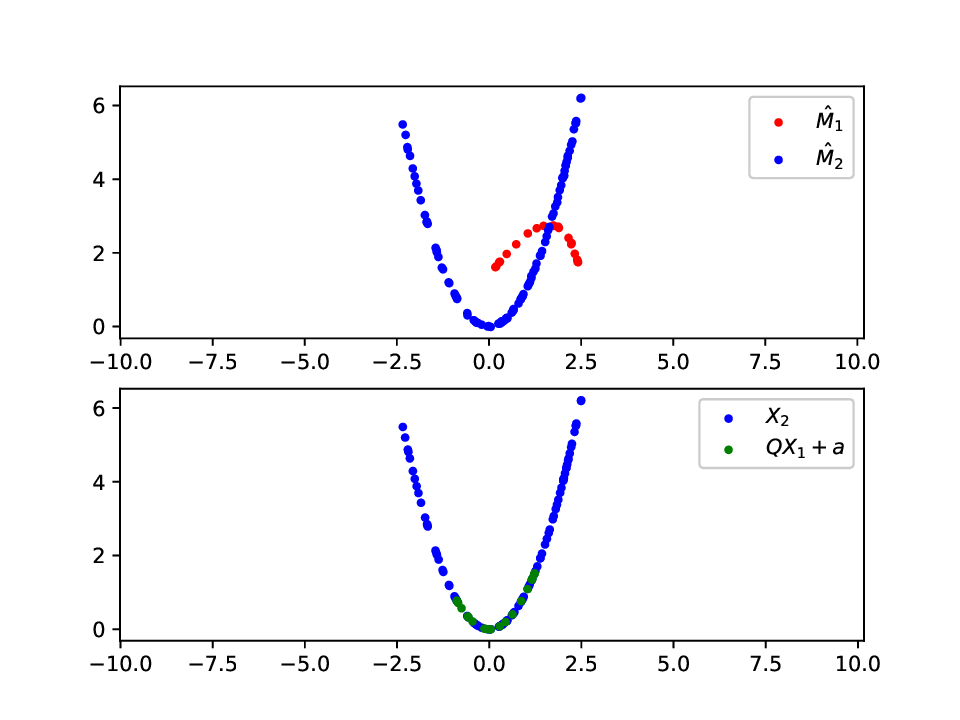}
\caption{Noisy registration with $\sigma = 10^{-2}$ and partial overlap}
\label{fig:registration_noise2_overlap}
\end{figure}

\begin{figure}[!htb]
\centering
\includegraphics[width = 0.7\textwidth]{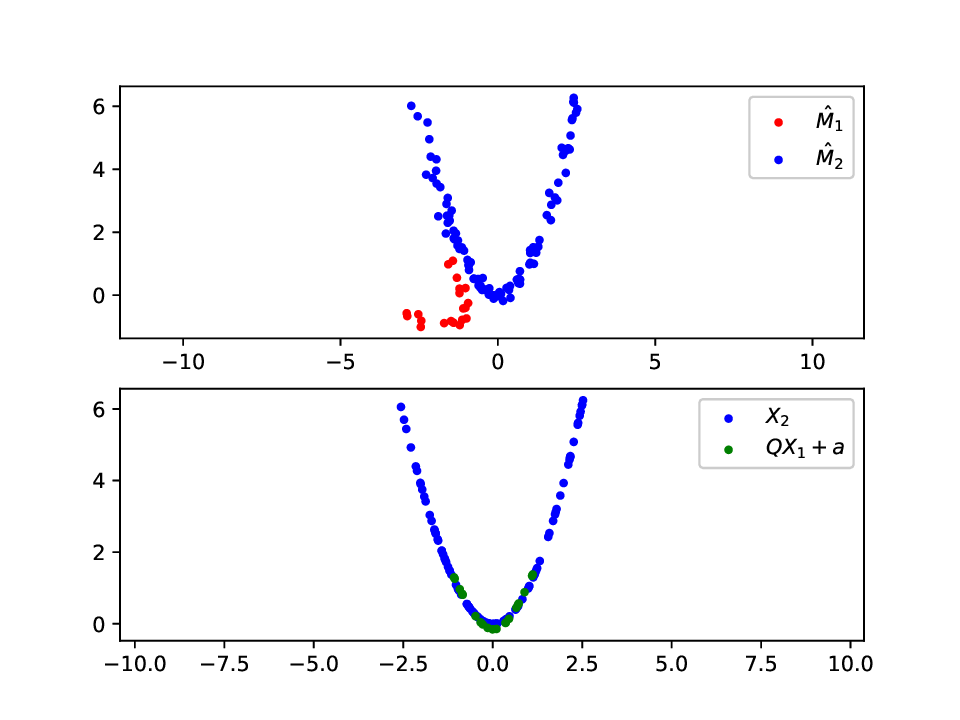}
\caption{Noisy registration with $\sigma = 10^{-1}$ and partial overlap}
\label{fig:registration_noise1_overlap}
\end{figure}

\paragraph*{Example 9: no overlap}
In Figure~\ref{fig:registration_noise_nooverlap}, we push things even further and show that the registration may be possible even in cases where there is no overlap between the point clouds. Our approach makes this possible because the points belong to a common algebraic variety. The noise level is set to $\sigma = 10^{-2}$ and $\textrm{RESIDUAL}  \approx 10^{-4}$.

\begin{figure}[!htb]
\centering
\includegraphics[width = 0.7\textwidth]{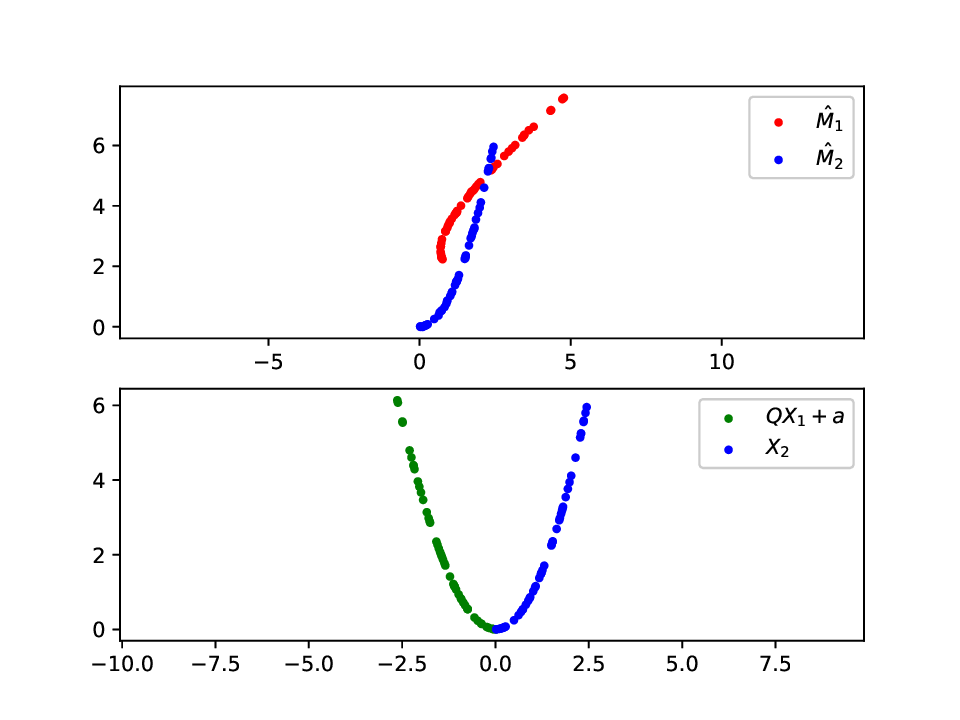}
\caption{Noisy registration with $\sigma = 10^{-2}$ and no overlap}
\label{fig:registration_noise_nooverlap}
\end{figure}

\paragraph*{Example 10: registration of dental scan}
We consider a computed tomography (CT) dental scan of dimension $3 \times 2048$. The data is naturally noisy and we use the algebraic variety model as an approximation. Nevertheless, we are able to estimate the transformation between two identical versions of the scan with some level of accuracy. The degree which yields the best result for the monomial features is $d=2$ and the output of the solver gives $\textrm{RESIDUAL} \approx 10^{-1}$. 
\begin{figure}[!htb]
\centering
\includegraphics[width = 1\textwidth]{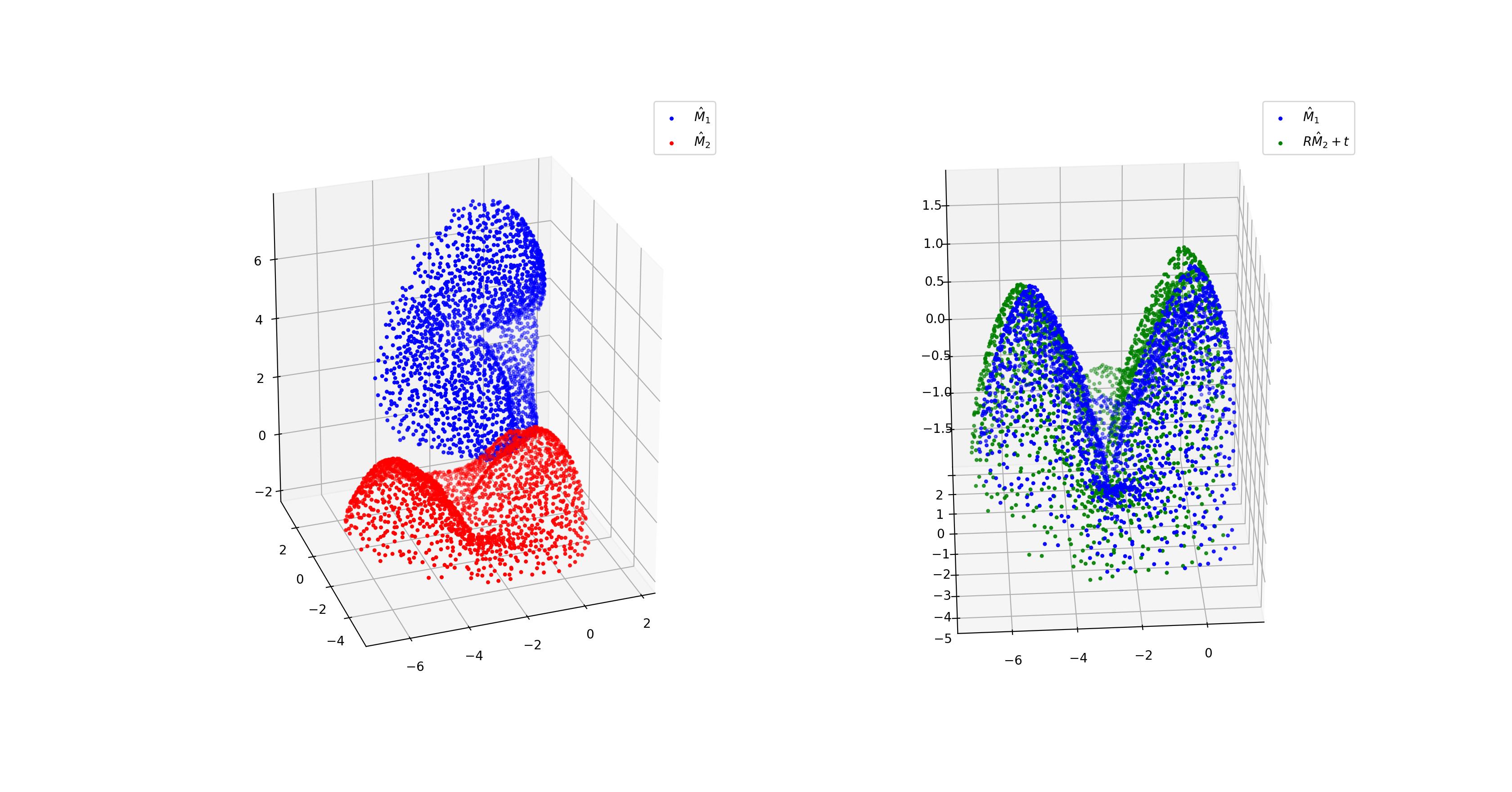}
\caption{Registration for dental scan.}
\end{figure}
\FloatBarrier

\section{Conclusions}
We present a framework which uses polynomial feature maps to approximate point clouds by algebraic varieties and perform point cloud registration based on these approximations. The approach is a conceptually appealing way to perform denoising on algebraic varieties and numerical results show high accuracy in the approximation for various noise levels, with a theoretical estimate of the accuracy using Stein's unbiased risk estimate. We identify the polynomial equations that define the algebraic variety from a set of noisy samples. The use of second-order Riemannian optimization methods allows to achieve high accuracy.

For the registration problem, we show numerical evidence that we are able to overlap point clouds which are approximated by algebraic varieties. We observe some robustness to noise and the method performs even on challenging data such as dental scans. An advantage of our approach over common algorithms for rigid registration is that we do not assume exact point matching and our method achieves good results even in cases of partial overlap between the two point clouds. 
  
 Our framework is so far best suited to low-dimensional data sets, because the dimension of the feature space $N\approx n^d/d!$ increases rapidly with $n$ and $d$; and the number of data points must be large enough to satisfy~\aref{assu:samples}. Several directions should be investigated to improve the scaling of these methods. The dimension of the feature space could be reduced using low-dimensional representations of the features, in the spirit of~\citep{rahimi2007random}. To increase the degree---for challenging surfaces---it would be important to use orthogonal polynomial basis to avoid the ill-conditioning of the monomials. Instead of increasing the degree of the approximation, it may be possible to divide the point cloud in patches and apply an approximation by an algebraic variety of degree $2$ on each patch. 

The scalability of the algorithm with respect to the number of data points $s$ may be improved by viewing the cost functions in~\eqref{eq:lambda_prob} and~\eqref{eq:regi_noisy} as a finite sum over the data points and using stochastic optimization methods~\citep{kohler2017subsampled,xu2020newton}. 

\paragraph{Acknowledgement}
The authors would like to thank Wenjuan Sun from the National Physical Laboratory (UK) who suggested the point cloud alignment problem that partially motivated this research and for providing the dental scan data.



\bibliographystyle{apalike} 
\bibliography{/Users/florentin/Documents/library/zotero2.bib}   

\end{document}